\documentclass[12pt,a4paper,reqno]{amsart}
\usepackage{amssymb}
\usepackage{verbatim}
\usepackage{amscd}
\usepackage{enumerate}
\usepackage{graphicx}
\usepackage{siunitx}
\usepackage{color}
\numberwithin{equation}{section}

\usepackage{mathabx}

\usepackage{mathtools}
\usepackage[tableposition=top]{caption}
\usepackage{booktabs,dcolumn}

\newcommand{\I}{\mathbb{I}} 
\newcommand{\E}{\mathbb{E}} 
\renewcommand{\P}{\mathbb{P}} 

\newcommand{\X}{\mathbf{X}} 
\newcommand{\Y}{\mathbf{Y}} 
\newcommand{\x}{\mathbf{x}} 
 
\newcommand{\n}{\mathbf{n}} 
\newcommand{\g}{\mathbf{g}} 
\renewcommand{\i}{\mathbf{i}} 

\theoremstyle{plain}

\newtheorem{theorem}{Theorem}[section]
\newtheorem{proposition}[theorem]{Proposition}

\newtheorem{lemma}[theorem]{Lemma}

\newtheorem{conjecture}[theorem]{Conjecture}

\theoremstyle{definition}

\newtheorem{remark}[theorem]{Remark}

\newcommand\R{\mathbb{R}}
\newcommand\Z{\mathbb{Z}}
\newcommand\N{\mathbb{N}}
\newcommand\C{\mathbb{C}}

\newcommand\eps{\varepsilon}

\begin{document}

\title[Chowla and Sarnak conjectures]{Equivalence of the logarithmically averaged Chowla and Sarnak conjectures}

\author{Terence Tao}
\address{Department of Mathematics, UCLA\\
405 Hilgard Ave\\
Los Angeles CA 90095\\
USA}
\email{tao@math.ucla.edu}

\begin{abstract}  Let $\lambda$ denote the Liouville function.  The Chowla conjecture asserts that
$$ \sum_{n \leq X} \lambda(a_1 n + b_1) \lambda(a_2 n+b_2) \dots \lambda(a_k n + b_k) = o_{X \to \infty}(X) $$
for any fixed natural numbers $a_1,a_2,\dots,a_k$ and non-negative integer $b_1,b_2,\dots,b_k$ with $a_ib_j-a_jb_i \neq 0$ for all $1 \leq i < j \leq k$, and any $X \geq 1$.  This conjecture is open for $k \geq 2$.  As is well known, this conjecture implies the conjecture of Sarnak that
$$ \sum_{n \leq X} \lambda(n) f(n) = o_{X \to \infty}(X)$$
whenever $f \colon \N \to \C$ is a fixed deterministic sequence and $X \geq 1$.  In this paper, we consider the weaker logarithmically averaged versions of these conjectures, namely that
$$ \sum_{X/\omega \leq n \leq X} \frac{\lambda(a_1 n + b_1) \lambda(a_2 n+b_2) \dots \lambda(a_k n + b_k)}{n} = o_{\omega \to \infty}(\log \omega) $$
and
$$ \sum_{X/\omega \leq n \leq X} \frac{\lambda(n) f(n)}{n} = o_{\omega \to \infty}(\log \omega)$$
under the same hypotheses on $a_1,\dots,a_k,b_1,\dots,b_k$ and $f$, and for any $2 \leq \omega \leq X$.  Our main result is that these latter two conjectures are logically equivalent to each other, as well as to the ``local Gowers uniformity'' of the Liouville function.  The main tools used here are the entropy decrement argument of the author used recently to establish the $k=2$ case of the logarithmically averaged Chowla conjecture, as well as the inverse conjecture for the Gowers norms, obtained by Green, Ziegler, and the author.
\end{abstract}

\maketitle


\section{Introduction}

Let $\lambda$ denote the Liouville function, thus $\lambda$ is the completely multiplicative function such that $\lambda(p)=-1$ for all primes $p$.  We have the following well known conjecture of Chowla \cite{chowla}:

\begin{conjecture}[Chowla conjecture]\label{chow}  Let $k \geq 1$, let $a_1,\dots,a_k$ be natural numbers and let $b_1,\dots,b_k$ be distinct nonnegative integers such that $a_i b_j - a_j b_i \neq 0$ for $1 \leq i < j \leq k$.  Then
$$ \sum_{n \leq X} \lambda(a_1 n + b_1) \dots \lambda(a_k n + b_k) = o_{X \to \infty}(X)$$
for all $X \geq 1$.  (See Section \ref{notation-sec} below for our asymptotic notation conventions.)
\end{conjecture}

Note that the bound of $o_{X \to \infty}(X)$ improves slightly over the trivial bound of $O(X)$.  The conjectures discussed later in this introduction will also similarly claim a slight improvement (of ``little-$o$'' type) over the corresponding trivial bound.

The $k=1$ case of the Chowla conjecture is equivalent to the prime number theorem.  The higher $k$ cases are open, although there are a number of partial results available if one allows for some averaging in the $b_1,\dots,b_k$ parameters, or if one wishes to obtain an upper bound in magnitude of the form $(1-\eps+o(1))X$ rather than $o(X)$; see \cite{mr}, \cite{mrt}, \cite{FH}, \cite{ey} for some recent results in this direction.  A routine application of the identity $\mu(n) = \sum_{d^2|n} \mu(d) \lambda(\frac{n}{d^2})$ (or the inverse identity $\lambda(n) = \sum_{d^2|n} \mu(\frac{n}{d^2})$) allows one to replace the Liouville function $\lambda$ in Conjecture \ref{chow} by the M\"obius function $\mu$ if desired; see e.g. \cite[\S 6]{gt-mobius} for a closely related argument.  See also \cite{hil}, \cite{hpw}, \cite{mrt-2} for some results on the related topic of sign patterns for the Liouville function.

In \cite{sarnak}, \cite{sarnak-2}, Sarnak introduced the following related conjecture.
Recall that a \emph{topological dynamical system} $(Y,T)$ is a compact metric space $Y$ with a homeomorphism $T \colon Y \to Y$, and the \emph{topological entropy} $h(Y,T)$ of such a system is defined as 
$$ h(Y,T) := \lim_{\eps \to 0} \limsup_{n \to \infty} \frac{1}{n} \log N(\eps,n)$$
where $N(\eps,n)$ is the largest number of $\eps$-separated points in $Y$ using the metric $d_n:Y \times Y \to \R^+$ defined by
$$ d_n(x,y) := \max_{0 \leq i \leq n} d(T^i x, T^i y).$$
A sequence $f\colon \Z \to \C$ is said to be \emph{deterministic} if it is of the form
$$ f(n) = F(T^n x_0)$$
for all $n$ and some topological dynamical system $(Y,T)$ of zero topological entropy $h(Y,T)=0$, a base point $x_0 \in Y$, and a continuous function $F\colon Y \to \C$.

\begin{conjecture}[Sarnak conjecture]\label{sar}  Let $f\colon \N \to \C$ be a deterministic sequence.  Then
$$ \sum_{n \leq X} \lambda(n) f(n) = o_{X \to \infty}(X)$$
for all $X \geq 1$.
\end{conjecture}

Both Conjecture \ref{chow} and Conjecture \ref{sar} can be viewed as instances of the ``M\"obius pseudorandomness principle'' (see e.g. \cite[\S 13]{ik}).  In \cite{sarnak} it was observed that Conjecture \ref{sar} was implied by Conjecture \ref{chow}; see \cite{tao-sarnak}, \cite{ekld} for some proofs of this implication.  The Sarnak conjecture has been verified for many particular instances of zero entropy topological dynamical systems \cite{borg-1}, \cite{borg-2}, \cite{bst}, \cite{dart}, \cite{davenport}, \cite{ddm}, \cite{dk}, \cite{drm}, \cite{ekl}, \cite{ekld}, \cite{elr}, \cite{elr-2}, \cite{fklm}, \cite{fm}, \cite{green}, \cite{gt-mobius}, \cite{hanna}, \cite{kara}, \cite{kl}, \cite{ls}, \cite{mriv}, \cite{mriv-2}, \cite{mullner}, \cite{peck}, \cite{sarnak-ubis}, \cite{veech}; for further variants of the Sarnak conjecture, see \cite{eis}, \cite{ekld}, \cite{huang}.

Recently in \cite{tao-chowla}, we introduced the following logarithmically averaged version of Conjecture \ref{chow}:

\begin{conjecture}[Logarithmically averaged Chowla conjecture]\label{chow-log}  Let $k \geq 1$, let $a_1,\dots,a_k$ be natural numbers and let $b_1,\dots,b_k$ be distinct nonnegative integers such that $a_i b_j - a_j b_i \neq 0$ for $1 \leq i < j \leq k$.  Then one has
\begin{equation}\label{non}
\sum_{X/\omega \leq n \leq X} \frac{\lambda(a_1 n + b_1) \dots \lambda(a_k n + b_k)}{n} = o_{\omega \to \infty}(\log \omega) 
\end{equation}
for all $2 \leq \omega \leq X$.
\end{conjecture}

We bound $\omega$ from below by $2$ rather than $1$ to avoid the minor inconvenience of $\log \omega$ vanishing.  A standard averaging argument shows that Conjecture \ref{chow} implies Conjecture \ref{chow-log} for any fixed choice of $k$.  Conversely, if we could prove Conjecture \ref{chow-log} for $\omega>1$ fixed and an error term of $o_{X \to \infty}(1)$ instead of $o_{\omega \to \infty}(\log \omega)$, one could establish Conjecture \ref{chow} by a summation by parts argument.  We leave the details of these (routine) arguments to the interested reader.

By introducing the \emph{entropy decrement argument}, we were able to establish the $k=2$ case of Conjecture \ref{chow-log} in \cite{tao-chowla}; using this result (or more precisely, a generalisation of this result in which $\lambda$ is replaced by a more general bounded completely multiplicative function, in the spirit of the Elliott conjecture \cite{elliott}), we were able to affirmatively settle the Erd\H{o}s discrepancy problem \cite{tao-erd}.

One can of course restrict this conjecture to the model case $a_1=\dots=a_k=1$:

\begin{conjecture}[Logarithmically averaged Chowla conjecture, special case]\label{chow-log-2}  Let $k \geq 1$, and let $h_1 < \dots < h_k$ be distinct nonnegative integers.  Then
\begin{equation}\label{non-2}
 \sum_{X/\omega \leq n \leq X} \frac{\lambda(n + h_1) \dots \lambda(n + h_k)}{n} = o_{\omega \to \infty}(\log \omega) 
\end{equation}
for all $2 \leq \omega \leq X$.
\end{conjecture}

We also have a logarithmically averaged version of the Sarnak conjecture:

\begin{conjecture}[Logarithmically averaged Sarnak conjecture]\label{sar-log}  Let $f\colon \N \to \C$ be a deterministic sequence.
Then
\begin{equation}\label{logsar}
 \sum_{X/\omega \leq n \leq X} \frac{\lambda(n) f(n)}{n} = o_{\omega \to \infty}(\log \omega)
\end{equation}
for all $2 \leq \omega \leq X$.
\end{conjecture}

We introduce two further conjectures which will be relevant in the proof of our main theorem.  Recall that for any finitely supported function $f\colon \Z \to \C$ and any $d \geq 1$, the \emph{Gowers uniformity norm} $\|f\|_{U^d(\Z)}$, first introduced in \cite{gowers-4}, \cite{gowers}, is defined by the formula
$$\|f\|_{U^d(\Z)} := \left( \sum_{x,h_1,\dots,h_d \in \Z} \prod_{\vec \omega \in \{0,1\}^d} {\mathcal C}^{|\vec \omega|} f(x+\omega_1 h_1 + \dots + \omega_d h_d) \right)^{1/2^d},$$ 
where $\vec \omega = (\omega_1,\dots,\omega_d)$, $|\vec \omega| := \omega_1+\dots+\omega_d$, and ${\mathcal C}\colon z \mapsto \overline{z}$ is the complex conjugation operator.  One can verify that $\|f\|_{U^d(\Z)}$ is well-defined as a non-negative real.  Given a non-empty discrete interval $I$ in the integers $\Z$, we define the local Gowers norm $\|f\|_{U^d(I)}$ by the formula
$$ \|f\|_{U^d(I)} := \| f 1_I \|_{U^d(\Z)} / \| 1_I \|_{U^d(\Z)}$$
where $1_I$ is the indicator function of $I$.  We then form the following conjecture:

\begin{conjecture}[Logarithmically averaged local Gowers uniformity of Liouville]\label{lgi-log}  Let $d \geq 1$.  Then one has
\begin{equation}\label{loggow}
 \sum_{X/\omega \leq n \leq X} \frac{\| \lambda \|_{U^d([n,n+H] \cap \Z)}}{n} = o_{H \to \infty}(\log \omega)
\end{equation}
for all $2 \leq H \leq \omega \leq X$.
\end{conjecture}

The constraint $H \leq \omega$ is mainly for aesthetic convenience (otherwise one would have to replace the $o_{H \to \infty}(\log \omega)$ term on the right-hand side with $o_{H \to \infty}(\log \omega)+o_{\omega \to \infty}(\log \omega)$); in any event, the conjecture is strongest and most interesting in the regime where $H$ is small compared with $X$.  The $d=1$ form of this conjecture follows from the recent breakthrough work of Matomaki and Radziwi{\l}{\l} \cite{mr}, but the $d>1$ cases remain open.  However, when one considers the regime where $\omega$ is fixed and $H$ is large, the results in \cite{gtz}, \cite{gt-mobius} give the claim \eqref{loggow} when $H \geq X$, and when $d=2$ the results of \cite{zhan} extend this to $H \geq X^{5/8+\eps}$ for any fixed $\eps>0$.

The Gowers norms are known to be connected to a special type of deterministic sequence, namely the \emph{nilsequences}, through the \emph{inverse conjecture for the Gowers norms}, proven in \cite{gtz} after building on prior work in \cite{gowers-4}, \cite{gowers}, \cite{gt-inverseu3}, \cite{gtz-4}.  As we shall see later in this paper, this result shows that Conjecture \ref{lgi-log} can be placed in the following equivalent form.  Recall that an \emph{$s$-step nilmanifold} is a manifold of the form $G/\Gamma$ where $G$ is a connected, simply connected nilpotent Lie group of step $s$, and $\Gamma$ is a cocompact discrete subgroup of $G$.  We can give such a manifold a smooth Riemannian metric for the purpose of defining concepts such as a Lipschitz function on $G/\Gamma$; we will not specify the exact choice of this metric as any two such metrics are equivalent.  The topological dynamical systems $(G/\Gamma, x \mapsto gx)$ for $g \in G$ are known as \emph{nilsystems}, and sequences of the form $n \mapsto F(g^n x_0)$ for some continuous $F\colon G / \Gamma \to \C$, group element $g \in G$, and base point $x_0 \in G/\Gamma$ are known as (basic) \emph{nilsequences}.  It is not difficult to show that nilsystems have zero topological entropy, and hence all nilsequences are deterministic.

\begin{conjecture}[Logarithmically averaged local Liouville-nilsequences conjecture]\label{lnc-log}  Let $s \geq 0$.  Let $G/\Gamma$ be an $s$-step nilmanifold, let $F\colon G/\Gamma \to \C$ be Lipschitz continuous, and let $x_0 \in G/\Gamma$.  Then
\begin{equation}\label{lognil}
 \sum_{X/\omega \leq n \leq X} \frac{\sup_{g \in G} |\sum_{h=1}^H \lambda(n+h) F( g^{h} x_0 )|}{n}
 = o_{H \to \infty}(H \log \omega).
\end{equation}
for all $2 \leq H \leq \omega \leq X$.
\end{conjecture}

Note carefully that the supremum in $g$ here is \emph{inside} the summation in $n$.  Analogously with the preceding conjecture, the $s=0$ case of this conjecture was established in \cite{mr}, but the $s \geq 1$ cases remain open.  As with Conjecture \ref{lgi-log}, in the regime where $\omega$ is fixed and $H$ is large, the results in \cite{gt-mobius} give the above claim for $H \geq X$, and when $s=1$ the results of Zhan \cite{zhan} extend this to $H \geq X^{5/8+\eps}$.  A variant of the $s=1$ case of Conjecture \ref{lnc-log}, in which the supremum in $g,x_0$ is placed outside the summation in $n$, but $\omega$ can be taken to be independent of $x$, was established in \cite{mrt}.

We are now ready to state the main result of this paper.

\begin{theorem}\label{main}  Conjectures \ref{chow-log}, \ref{chow-log-2}, \ref{sar-log}, \ref{lgi-log}, \ref{lnc-log} are equivalent.
\end{theorem}

\begin{remark}  An inspection of the arguments in this paper reveals that all of the equivalences in this theorem continue to hold if we enforce a fixed functional relationship between $\omega$ and $X$.  For instance, choosing the relationship $X=\omega$, we can show the equivalence of the logarithmically averaged Chowla conjecture
$$
 \sum_{n \leq X} \frac{\lambda(n + h_1) \dots \lambda(n + h_k)}{n} = o_{X \to \infty}(\log X) $$
for all fixed distinct natural numbers $h_1,\dots,h_k$, with the logarithmically averaged Sarnak conjecture
$$ \sum_{n \leq X} \frac{\lambda(n) f(n)}{n} = o_{X \to \infty}(\log X)$$
for all fixed deterministic sequences $f$.   
\end{remark}

We summarise the key implications in this theorem as follows (see Figure \ref{fig:equivs}):

\begin{figure} [t]
\centering
\includegraphics{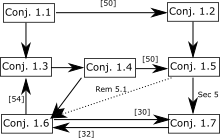}
\caption{Logical implications between conjectures, annotated by the reference or section where the implication (or some minor variant of that implication) is essentially proven.  Implications without any annotation are trivial. The dotted arrow refers to the potential implication sketched in Remark \ref{Entr}.  One could enlarge this diagram by adding non-logarithmically-averaged versions of Conjectures \ref{chow-log-2}, \ref{lgi-log}, \ref{lnc-log}; we leave this task to the interested reader.}
\label{fig:equivs}
\end{figure}

\begin{itemize}
\item The implication of Conjecture \ref{chow-log-2} from Conjecture \ref{chow-log} is trivial.
\item The implication of Conjecture \ref{sar-log} from Conjecture \ref{chow-log-2} was essentially already observed in \cite{sarnak}, but for the convenience of the reader we give a self-contained derivation in Section \ref{chow-sec}.
\item The derivation of Conjecture \ref{chow-log} from Conjecture \ref{lgi-log} follows from adapting the entropy decrement argument in \cite{tao-chowla}, and is given in Section \ref{entropy-dec}.
\item The derivation of Conjecture \ref{lgi-log} from Conjecture \ref{lnc-log} follows from the inverse conjecture for the Gowers norms \cite[Theorem 1.1]{gtz}, and is given in Section \ref{gow-sec}.  (The converse implication is proven similarly using the converse \cite[Proposition 12.6]{gt-inverseu3} to the inverse conjecture, which is much easier to prove.)
\item Finally, the derivation of Conjecture \ref{lnc-log} from Conjecture \ref{sar-log} follows from an estimation of the metric entropy of the space of nilsequences of controlled complexity, and is morally (though not quite) a consequence of the zero-entropy nature of nilsystems; we detail this in Section \ref{detsec}.
\end{itemize}

\begin{remark} Most of the arguments in this paper should extend if one replaces the Liouville function by a more general bounded multiplicative function; the main obstruction to this is that one would now need some sort of ``higher order restriction theorem for the primes'' in the entropy decrement step (used to deduce Conjecture \ref{chow-log} from Conjecture \ref{lgi-log}), generalising the ``linear restriction theorem'' used in \cite[Lemma 3.7]{tao-chowla}.  We will not pursue this matter here.
\end{remark}

\begin{remark}  The implication of Conjecture \ref{chow-log} from Conjecture \ref{lgi-log} is the only part of the argument that requires the logarithmic averaging; all of the other implications are valid if Conjectures \ref{chow-log}, \ref{chow-log-2}, \ref{sar-log}, \ref{lgi-log}, \ref{lnc-log} are replaced by their non-logarithmically averaged counterparts (such as Conjecture \ref{chow} or Conjecture \ref{sar}).
\end{remark}

\begin{remark} In addition to the above implications, there is also an easy way to deduce Conjecture \ref{lgi-log} from Conjecture \ref{chow-log-2}.  Indeed, from expanding out the Gowers norms and interchanging summations, we see from Conjecture \ref{chow-log-2} that
$$
 \sum_{X/\omega \leq n \leq X} \frac{\| \lambda \|_{U^d([n,n+H] \cap \Z)}^{2^d}}{n} = o_{H \to \infty}(\log \omega)
$$
if $H$ is sufficiently slowly growing as a function of $\omega$, which by H\"older's inequality gives Conjecture \ref{lgi-log} in the case when $H$ is sufficiently slowly growing; one can then use the Gowers-Cauchy-Schwarz inequality \cite{gowers} to control the Gowers norms for large values of $H$ in terms of Gowers norms for small values of $H$, giving Conjecture \ref{lgi-log} in general; we leave the details to the interested reader.  See also Remark \ref{Entr} for another possible implication that avoids the use of the (difficult) inverse conjecture
for the Gowers norms.
\end{remark}

\subsection{Notation}\label{notation-sec}

We adopt the usual asymptotic notation of $A \ll B$, $B \gg A$, or $A = O(B)$ to denote the assertion that $|A| \leq CB$ for some constant $C$.  If we need $C$ to depend on an additional parameter we will denote this by subscripts, e.g. $A = O_\eps(B)$ denotes the bound $|A| \leq C_\eps B$ for some $C_\eps$ depending on $\eps$.  

In all of our results, there will be a number of asymptotic parameters such as $X, \omega, H$, as well as ``fixed'' quantities (such as $k$, $f$, $d$, $a_1,\dots,a_k$, $b_1,\dots,b_k$) that do not depend on the asymptotic parameters; the distinction should be clear from context.  (In particular, in each of the conjectures stated in the introduction, the variables introduced before the word ``Then'' are fixed, and the variables appearing afterwards are asymptotic parameters.) Given an asymptotic parameter such as $X$, we use $A = o_{X \to \infty}(B)$ to denote the bound $|A| \leq c(X) B$ where $c(X)$ depends only on $X$ and fixed quantities and goes to zero as $X \to \infty$ (subject to whatever restrictions are in place on the asymptotic parameters, such as $1 \leq H \leq \omega \leq X$).

If $E$ is a statement, we use $1_E$ to denote the indicator, thus $1_E=1$ when $E$ is true and $1_E=0$ when $E$ is false, and $1_A(x) = 1_{x \in A}$ for any set $A$ and point $x$.

Given a finite set $S$, we use $|S|$ to denote its cardinality.

For any real number $\alpha$, we write $e(\alpha) := e^{2\pi i \alpha}$; this quantity lies in the unit circle $S^1 := \{ z \in \C: |z|=1\}$.  By abuse of notation, we can also define $e(\alpha)$ when $\alpha$ lies in the additive unit circle $\R/\Z$.

All sums and products will be over the natural numbers $\N =\{1,2,\dots\}$ unless otherwise specified, with the exception of sums and products over $p$ which is always understood to be prime.

We use $d|n$ to denote the assertion that $d$ divides $n$, and $n\ (d)$ to denote the residue class of $n$ modulo $d$.  

We will frequently use probabilistic notation such as the expectation $\E \X$ of a random variable $\X$ or a probability $\P(E)$ of an event $E$. We will use boldface symbols such as $\X$, $\Y$ or $\n$ to refer to random variables.

A particularly important random variable for us will be the following.  Suppose we are given some parameters $2 \leq \omega \leq X$.  We then define $\n$ to be the random natural number $\{n: X/\omega \leq n \leq X\}$ drawn with probability distribution
$$ \P(\n = n) := \frac{1/n}{\sum_{X/\omega \leq n \leq X} 1/n}.$$
Since $\sum_{X/\omega \leq n \leq X} 1/n$ is comparable to $\log \omega$, we can rewrite many of the logarithmically averaged claims conjectured in the introduction in probabilistic notation.  Specifically, the bound \eqref{non} may be rewritten as
\begin{equation}\label{non-equiv}
\E \lambda(a_1 \n + b_1) \dots \lambda(a_k \n + b_k) = o_{\omega \to \infty}(1),
\end{equation}
and similarly \eqref{non-2} may be rewritten as
\begin{equation}\label{non2-equiv}
\E \lambda(\n + h_1) \dots \lambda(\n + h_k) = o_{\omega \to \infty}(1).
\end{equation}
Continuing in this vein, \eqref{logsar} is equivalent to
\begin{equation}\label{logsar-equiv}
 \E \lambda(\n) f(\n) = o_{\omega \to \infty}(1),
\end{equation}
\eqref{loggow} is equivalent to
\begin{equation}\label{loggow-equiv}
 \E \| \lambda \|_{U^d([\n,\n+H] \cap \Z)} = o_{H \to \infty}(1)
\end{equation}
and \eqref{lognil} is equivalent to
\begin{equation}\label{lognil-equiv}
\E \sup_{g \in G} \left|\sum_{h=1}^H \lambda(\n+h) F( g^{\n} x_0 )\right| = o_{H \to \infty}(H).
\end{equation}

We will rely heavily on the following approximate affine invariance of the random variable $\n$:

\begin{lemma}[Approximate affine invariance] \label{aai} Let $q$ be a natural number, and let $r$ be an integer.  Suppose that $\omega$ is sufficiently large depending on $q,r$. Then for any complex-valued random variable $F(\n)$ depending on $\n$ and bounded in magnitude by $O(1)$, one has
$$ \E F(\n) 1_{\n = r\ (q)} = \frac{1}{q} \E F(q\n+r) + o_{\omega \to \infty}(1).$$
\end{lemma}

\begin{proof} See \cite[Lemma 2.5]{tao-chowla}.  (The statement there involved additional parameters $H_+$, $A$ intermediate between $q,r$ and $\omega$, but it is easy to see that one can delete these parameters from the statement and proof of that lemma.)
\end{proof}

Specialising this lemma to the case $q=1$, we obtain the approximate translation invariance
\begin{equation}\label{ati}
 \E F(\n) =  \E F(\n+r) + o_{\omega \to \infty}(1)
\end{equation}
when $\omega$ is sufficiently large depending on $r$.  This translation invariance will be sufficient for establishing the implications in Section \ref{chow-sec} and Section \ref{gow-sec}, but the argument in Section \ref{entropy-dec} requires the full affine invariance from Lemma \ref{aai}, which is only available in the logarithmically averaged setting.

\subsection{Acknowledgments}

The author is supported by NSF grant DMS-1266164 and by a Simons Investigator Award. The author also thanks Ben Green for comments and encouragement.

\section{From Chowla to Sarnak}\label{chow-sec}

In this section we deduce Conjecture \ref{sar-log} from Conjecture \ref{chow-log-2}.  Our arguments are an adaptation of those in \cite{tao-sarnak}.

Fix a topological dynamical system $(Y,T)$ of zero topological entropy, a base point $x_0 \in Y$, and a continuous function $F\colon Y \to \C$.  We allow all implied constants in the asymptotic notation to depend on these quantities.  
We introduce the following parameters:
\begin{itemize}
\item We let $\eps > 0$ be a quantity that is sufficiently small (depending on the fixed quantities $(Y,T),x_0,F$).
\item Then, we let $H$ be a quantity that is sufficiently large depending on $\eps$ (and the fixed quantities).
\item Finally, we let $2 \leq \omega \leq X$ be quantities with $\omega$ sufficiently large depending on $\eps,H$ (and the fixed quantities).
\end{itemize}

Let $\n$ be as in the previous section.  Using the form \eqref{logsar-equiv} of Conjecture \ref{sar-log}, we see that it will suffice to establish the bound
$$ \E \lambda(\n) F(T^\n x_0) \ll \eps $$
under the above assumptions on $\eps, H, \omega, X$.

From approximate translation invariance \eqref{ati}, we have
$$ \E \lambda(\n+h) F(T^{\n+h} x_0) = \E \lambda(\n) F(T^\n x_0) + o_{\omega \to \infty}(1)$$
for any $1 \leq h \leq H$, so in particular upon averaging in $h$ we obtain
$$ \E \frac{1}{H} \sum_{h=1}^H \lambda(\n+h) F(T^{\n+h} x_0) = \E \lambda(\n) F(T^\n x_0) + o_{\omega \to \infty}(1).$$
Thus it will suffice to show that
$$ \frac{1}{H} \sum_{h=1}^H \lambda(\n+h) F(T^{\n+h} x_0) \ll \eps $$
with probability $1-O(\eps)$, since this expression is already bounded by $O(1)$.

As $F$ is uniformly continuous, there exists $\delta>0$ depending on $\eps,F$ such that $|F(x)-F(y)| \leq \eps$ whenever $d(x,y) \leq \delta$.  As $(Y,T)$ has zero entropy, we see (if $H$ is large enough) that we can cover $Y$ by $O( \exp( \eps^3 H ) )$ balls of radius $\delta$ in the $d_H$ metric.  That is to say, we can find points $x_1,\dots,x_m \in Y$ with $m \ll \exp(\eps^3 H)$ such that for each $y \in Y$, there exists $1 \leq i \leq m$ such that
$$ d(T^h x_i, T^h y) \leq \delta$$
for all $1 \leq h \leq H$.  Applying this with $y$ replaced by $T^\n x_0$, we conclude that there exists a random variable $1 \leq \i \leq m$ such that
$$ d(T^h x_\i, T^{\n+h} x_0) \leq \delta$$
for all $1 \leq h \leq H$, and in particular
$$ \frac{1}{H} \sum_{h=1}^H \lambda(\n+h) F(T^{\n+h} x_0) = \frac{1}{H} \sum_{h=1}^H \lambda(\n+h) F(T^{h} x_\i) + O(\eps).$$
Thus it will suffice to show that
$$ \left|\sum_{h=1}^H \lambda(\n+h) F(T^{h} x_\i)\right| \leq \eps H$$
with probability $1-O(\eps)$.  Since there are only $O(\exp(\eps^3 H))$ choices for $\i$, it suffices by the union bound to show that
$$ \left|\sum_{h=1}^H \lambda(\n+h) F(T^{h} x_i)\right| \leq \eps H$$
with probability $1 - O( \exp( c \eps^2H ) )$ for some fixed $c>0$ and all (deterministic) $i=1,\dots,m$.  

Let $k \leq H/2$ be a natural number to be chosen later.  By the Chebyshev inequality, we have
\begin{equation}\label{nal}
\P\left( \left|\sum_{h=1}^H \lambda(\n+h) F(T^{h} x_i)\right| > \eps H \right) \leq (\eps H)^{-2k} \E \left|\sum_{h=1}^H \lambda(\n+h) F(T^{h} x_i)\right|^{2k}.
\end{equation}

On the other hand from Conjecture \ref{chow-log-2} (in the form \eqref{non2-equiv}), we have
\begin{equation}\label{lan}
 \E \lambda(\n+h_1) \dots \lambda(\n+h_{2k}) = o_{\omega \to \infty}(1)
\end{equation}
for any $1 \leq h_1 < \dots < h_{2k} \leq H$, since $\omega$ is assumed sufficiently large depending on $H$.

Expanding out the expression inside the expectation in \eqref{nal}, we obtain $H^{2k}$ terms, most of which are $o_{\omega \to \infty}(1)$ thanks to \eqref{lan}.  The cumulative contribution of all such terms to \eqref{nal} is still $o_{\omega \to \infty}$, since $\omega$ is assumed large depending on $H$ (and hence on $k$).  The only terms which are not of this form are terms in which each factor of $\lambda(\n+h)$ occurs at least twice (so in particular at most $k$ different values of $h$ appear).  Crude counting shows that there are at most $k^{2k} \binom{H}{k} = O( Hk )^k$ such terms, each of which contributes at most $O(1)$ to the above sum, and hence
$$\P\left( \left|\sum_{h=1}^H \lambda(\n+h) F(T^{h} x_i)\right| > \eps H \right) \ll (\eps H)^{-2k} O(Hk)^k + o_{\omega \to \infty}(1).$$
Choosing $k$ to be a small multiple of $\eps^2 H$ (rounded to the nearest integer), we obtain the claim.

\section{The entropy decrement argument}\label{entropy-dec}

In this section we use the entropy decrement argument from \cite{tao-chowla}, together with some Cauchy-Schwarz type manipulations similar to that used in \cite{fhk}, \cite{wz}, as well as known results on linear equations on primes \cite{gt-primes}, to deduce Conjecture \ref{chow-log} from Conjecture \ref{lgi-log}.  

We first make some easy reductions in Conjecture \ref{chow-log}.  Firstly, we may assume $k > 2$, since the $k \leq 2$ case was already established in \cite{tao-chowla}.  Next, if we set $a := a_1 \dots a_k$, then $\lambda(a_i n + b_i)$ is a constant multiple of $\lambda(a n + b'_i)$, where $b'_i := a_1 \dots a_{i-1} b_i a_{i+1} \dots a_k$.  Thus (replacing $a_i,b_i$ with $a,b'_i$ for each $i$) we may assume without loss of generality that $a_1=\dots=a_k=a$, in which case the condition $a_ib_j-a_j b_i \neq 0$ now simplifies to the requirement that the $b_1,\dots,b_k$ are distinct.  

Henceforth $k,a,b_1,\dots,b_k$ are considered fixed.  We allow all implied constants in the argument below to depend on $k,a,b_1,\dots,b_k$.  We select some further quantities:

\begin{itemize}
\item First, we let $\eps>0$ be a quantity that is sufficiently small depending on $k,a,b_1,\dots,b_k$.
\item Then, we select a natural number $w$ that is sufficiently large depending on $k,a,b_1,\dots,b_k,\eps$.
\item Then, we select a natural number $H_-$ that is sufficiently large depending on $k,a,b_1,\dots,b_k,\eps,w$.
\item Then, we select a natural number $H_+$ that is sufficiently large depending on $k,a,b_1,\dots,b_k,\eps,w,H_-$.
\item Finally, we let $\omega, X$ be quantities such that $2 \leq \omega \leq X$ such that $\omega$ is sufficiently large depending on $k,a,b_1,\dots,b_k,\eps,w,H_-,H_+$.
\end{itemize}
The reader may find it convenient to keep the hierarchy
$$ 1 \ll \frac{1}{\eps} \ll w \ll H_- \ll H_+ \ll \omega \leq X $$
in mind in the arguments which follow.

Using the form \eqref{non-equiv}, it will now suffice to establish the bound
$$ \E \prod_{i=1}^k \lambda(a\n + b_i) \ll \eps.$$
Using approximate translation invariance \eqref{ati}, we may assume without loss of generality that $b_1=0$.

Assume for sake of contradiction that the claim failed, thus 
\begin{equation}\label{ao}
 \left|\E \prod_{i=1}^k \lambda(a\n + b_i)\right| \gg \eps.
\end{equation}

We now use Lemma \ref{aai} to convert the single average in \eqref{ao} to a double average, as in \cite[Proposition 2.6]{tao-chowla}:

\begin{proposition}\label{cloc}  Suppose that \eqref{ao} holds. Let $H_- \leq H \leq H_+$, and let ${\mathcal P}_H$ denote the set of primes between $\frac{\eps^2}{2} H$ and $\eps^2 H$.  Then
$$ \left|\E \sum_{p \in {\mathcal P}_H} \sum_j 1_{a\n + j = 0\ (ap)} \prod_{i=1}^k \lambda(a\n + j + p b_i) 1_{[1,H]}(j+pb_i)\right| \gg \eps \frac{H}{\log H}.$$
\end{proposition}

\begin{proof}  Write
$$ Q := \E 1_{\n = 0\ (a)} \prod_{i=1}^k \lambda(\n+b_i),$$
then \eqref{ao} and Lemma \ref{aai} implies that $|Q| \gg \eps$.  For any prime $p$, we have $\lambda(p)=-1$, and hence from the complete multiplicativity of the Liouville function we have the identity
$$ 1_{\n = 0\ (a)} \prod_{i=1}^k \lambda(\n+b_i) = (-1)^k 1_{p\n = 0\ (ap)} \prod_{i=1}^k \lambda(p\n+pb_i)$$
and thus
$$ \E 1_{p\n = 0\ (ap)} \prod_{i=1}^k \lambda(p\n+pb_i) = (-1)^k Q.$$
Applying Lemma \ref{aai} and noting that $1_{\n=0\ (ap)} 1_{\n=0\ (p)} = 1_{\n=0\ (ap)}$, we conclude that
$$ \E 1_{\n = 0\ (ap)} \prod_{i=1}^k \lambda(\n+pb_i) = (-1)^k \frac{Q}{p} + o_{\omega \to \infty}(1).$$
for any prime $p \leq H$.  Shifting $\n$ by $j$ using another application of Lemma \ref{aai}, we conclude that
$$ \E 1_{\n+j = 0\ (ap)} \prod_{i=1}^k \lambda(\n+j+pb_i) = (-1)^k \frac{Q}{p} + o_{\omega \to \infty}(1).$$
for any prime $p \leq H$ and any $1 \leq j \leq H$.  Summing in $j$, we conclude (recalling that $\omega$ is assumed large compared with $H_+$ and hence $H$)
$$ \E \sum_{j=1}^H 1_{\n+j = 0\ (ap)} \prod_{i=1}^k \lambda(\n+j+pb_i) = (-1)^k \frac{HQ}{p} + o_{\omega \to \infty}(1).$$
If we now introduce the quantity
$$ R(s) = R_p(s) := \E \sum_{j=1}^H 1_{\n+j = 0\ (ap)} \prod_{i=1}^k \lambda(\n+j+pb_i) 1_{\n = s\ (a)}$$
for $s \in \Z/a\Z$, we therefore have
\begin{equation}\label{sumq}
 \sum_{s \in\Z/a\Z} R(s) = (-1)^k \frac{HX}{p} + o_{\omega \to \infty}(1).
\end{equation}
On the other hand, applying Lemma \ref{aai} with $\n$ shifted to $\n+1$, and then shifting $j$ by one, we have
$$ R(s+1) := \E \sum_{j=2}^{H+1} 1_{\n+j = 0\ (ap)} \prod_{i=1}^k \lambda(\n+j+pb_i) 1_{\n = s\ (a)}.$$
The difference between $\sum_{j=2}^{H+1} 1_{\n+j = 0\ (ap)} \prod_{i=1}^k \lambda(\n+j+pb_i) 1_{\n = s\ (a)}$ and $\sum_{j=1}^H 1_{\n+j = 0\ (ap)} \lambda(\n+j+pb_1) \dots \lambda(\n+j+pb_k) 1_{\n = s\ (a)}$ is zero with probability $1-O(1/p)$, and $O(1)$ on the remaining event.  Absorbing the $o_{\omega \to \infty}(1)$ error into the $O(1/p)$ error, we conclude that
$$ R(s+1) = R(s) + O\left(\frac{1}{p}\right)$$
for all $s \in \Z/a\Z$, so $R$ fluctuates by at most $O(a/p)$.  Combining this with \eqref{sumq}, we conclude in particular that
$$ R(0) = (-1)^k \frac{HQ}{ap} + O\left( \frac{a}{p} \right).$$
Summing over ${\mathcal P}_H$, we conclude that
\begin{align*}
&\E \sum_{j=1}^H \sum_{p \in {\mathcal P}_H} 1_{\n+j = 0\ (ap)} \prod_{i=1}^k \lambda(\n+j+pb_i) 1_{\n = 0\ (a)}\\
&\quad = \left((-1)^k \frac{HQ}{a} + O(a)\right) \sum_{p \in {\mathcal P}_H} \frac{1}{p}
\end{align*}
and hence by the prime number theorem and the lower bound $|Q| \gg \eps$, we have
$$
\left|\E \sum_{j=1}^H \sum_{p \in {\mathcal P}_H} 1_{\n+j = 0\ (ap)} \prod_{i=1}^k \lambda(\n+j+pb_i) 1_{\n = 0\ (a)}\right| \gg \eps \frac{H}{a \log H}.$$
Applying Lemma \ref{aai}, we obtain
$$
\left|\E \sum_{j=1}^H \sum_{p \in {\mathcal P}_H} 1_{a\n+j = 0\ (ap)} \prod_{i=1}^k \lambda(a\n+j+pb_1)\right| \gg \eps \frac{H}{\log H}.$$
If one of the $j+pb_i$ lie outside of $[1,H]$, then $j$ lies in either $[1,B\eps^2 H]$ or $[(1-B\eps^2)H,H]$, where $B := \max(|b_1|,\dots,|b_k|)$.  The contribution of these values of $j$ can be easily estimated to be $O( \frac{\eps^2 B H}{\log H} )$, which is negligible since $\eps$ was assumed small.  Discarding these contributions, we obtain the proposition.
\end{proof}

We rewrite the conclusion of Proposition \ref{cloc} as
\begin{equation}\label{efy}
 |\E F(\X_H, \Y_H)| \gg \eps \frac{H}{\log H}
\end{equation}
where $\X_H$ is the discrete random variable
$$ \X_H := ( \lambda(a\n + j) )_{j=1,\dots,H}$$
(taking values in $\{-1,+1\}^{H}$), $\Y_H$ is the discrete random variable
$$ \Y_H := \n\ (P_H)$$
(taking values in $\Z/P_H\Z$) with $P_H := \prod_{p \in H} p$, and $F\colon \{-1,+1\}^{H} \times \Z/P_H\Z \to \R$ is the function
\begin{equation}\label{yando}
F( (x_{j})_{j=1,\dots,H}, y\ (P_H) ) :=
\sum_{p \in {\mathcal P}_H} \sum_j 1_{ay + j = 0\ (ap)} \prod_{i=1}^k x_{j+pb_1} 
\end{equation}
with the convention that $x_j=0$ for $j \not \in [1,H]$.

Crucially, we can locate a scale $H$ in which $\X_H$ and $\Y_H$ have a weak independence property:

\begin{proposition}[Entropy decrement argument]  There exists a natural number $H$ between $H_-$ and $H_+$ which is a multiple of $a$, such that
$$ \I( \X_H, \Y_H ) \leq \frac{H}{\log H \log\log\log H},$$
where $\I(\X_H,\Y_H)$ denotes the mutual information between $\X_H$ and $\Y_H$ (see \cite[\S 3]{tao-chowla} for a definition).
\end{proposition}

\begin{proof} See \cite[Lemma 3.2]{tao-chowla}.
\end{proof}

Let $H$ be as in the above proposition.  Repeating the derivation of \cite[(3.16)]{tao-chowla} (using in particular the Hoeffding concentration inequality \cite{hoeff}) almost verbatim, we may now conclude from \eqref{efy} that
$$ \left| \E \frac{1}{P_H} \sum_{y \in \Z/P_H\Z} F( \X_H, y ) \right| \gg \eps \frac{H}{\log H}.$$
But from the Chinese remainder theorem and \eqref{yando}, the left-hand side can be written as
$$ \left| \E \sum_{p \in {\mathcal P_H}} \frac{1}{p} \sum_{j = 0\ (a)} \prod_{i=1}^k \lambda(a \n + j + pb_i) 1_{[1,H]}(j+pb_i) \right|.$$
Writing $1 = \frac{\log p}{\log H} + O_\eps( \frac{1}{\log H} )$ and discarding the error term by the triangle inequality and prime number theorem, we thus have
$$ \left| \E \sum_{p \in {\mathcal P_H}} \frac{\log p}{p} \sum_{j = 0\ (a)} \prod_{i=1}^k \lambda(a \n + j + pb_i) 1_{[1,H]}(j+pb_i) \right| \gg \eps H.$$
If we let $\Lambda$ denote the von Mangoldt function, we thus have
\begin{equation}\label{aaa}
 \left| \E \sum_{\frac{\eps^2}{2} H \leq m \leq \eps^2 H} \frac{\Lambda(m)}{m} \sum_{j = 0\ (a)} \prod_{i=1}^k \lambda(a \n + j + pb_i) 1_{[1,H]}(j+pb_i)  \right| \gg \eps H,
\end{equation}
since the contribution of those $m$ which are powers of primes, rather than primes, is easily seen to be negligible.

It is now convenient to use the ``$W$-trick'' from \cite{gt-longaps}.  We recall the parameter $w$ introduced (but not yet used) at the beginning of the argument.  We set
$$ W := \prod_{p \leq w} p$$
and observe that the contribution to \eqref{aaa} of those $m$ that share a common factor with $W$ is negligible.  Discarding these terms and applying the pigeonhole principle, we conclude the existence of a natural number $1 \leq r \leq W$ coprime with $W$, such that
\begin{align*}
&\left| \E \sum_{\frac{\eps^2}{2} H \leq m \leq \eps^2 H} \frac{\Lambda(m) 1_{m = r\ (W)}}{m} \sum_{j = 0\ (a)} \prod_{i=1}^k \lambda(a \n + j + mb_i) 1_{[1,H]}(j+mb_i) \right|\\
&\quad \gg \eps \frac{H}{\phi(W)},
\end{align*}
where $\phi(W)$ is the Euler totient function of $W$.  Making the substitution $m = Wm' + r$, and discarding some negligible error terms, we conclude that
\begin{align*}
&\left| \E \sum_{\frac{\eps^2}{2}\frac{H}{W} \leq m \leq \eps^2 \frac{H}{W}} \frac{\Lambda(Wm+r)}{W m} \sum_{j = 0\ (a)} \prod_{i=1}^k \lambda(a \n + j + (Wm+r)b_i) 1_{[1,H]}(j+(Wm+r)b_i) \right| \\
&\quad\gg \eps \frac{H}{\phi(W)},
\end{align*}
so if we define
$$ \Lambda_{W,r}(m) := \frac{\phi(W)}{W} \Lambda(Wm+r)$$
then
\begin{equation}\label{thumb}
\begin{split}
&\left| \E \sum_{\frac{\eps^2}{2} \frac{H}{W} \leq m \leq \eps^2\frac{H}{W}} \frac{\Lambda_{W,r}(m)}{m} \sum_{j = 0\ (a)} \prod_{i=1}^k \lambda(a \n + j + (Wm+r)b_i) 1_{[1,H]}(j+(Wm+r)b_i)\right| \\
&\quad\gg \eps H.
\end{split}
\end{equation}
We now replace $\Lambda_{W,r}$ by $1$.  Manipulations of this form have appeared in \cite{fhk}, \cite{wz}; we will use an argument somewhat similar to that in \cite{fhk}:

\begin{proposition}[Elmination of von Mangoldt weight]  We have
\begin{align*}
& \E \sum_{\frac{\eps^2}{2} \frac{H}{W} \leq m \leq \eps^2 \frac{H}{W}} \frac{\Lambda_{W,r}(m)-1}{m} \sum_{j = 0\ (a)}
\prod_{i=1}^k \lambda(a \n + j + (Wm+r)b_i) 1_{[1,H]}(j+(Wm+r)b_i) \\
&\quad = o_{w \to \infty}(H).
\end{align*}
\end{proposition}

\begin{proof}  By the triangle inequality, it suffices to show the deterministic estimate
$$
\sum_{\frac{\eps^2}{2} \frac{H}{W} \leq m \leq \eps^2 \frac{H}{W}} \frac{\Lambda_{W,r}(m)-1}{m} \sum_j \prod_{i=1}^k f_i(j + (Wm+r)b_i) = o_{w \to \infty}(H)$$
for any functions $f_1,\dots,f_k\colon \Z \to [-1,1]$ supported on $[1,H]$ (note that the constraint $j=0\ (a)$ can be absorbed into (say) the $f_1$ factor). By shifting each $f_i$ by $rb_i$ (and restricting back to $[1,H]$ at the cost of a negligible error), we may replace each term $f_i(j+(Wm+r)b_i)$ here by $f_i(j+Wmb_i)$.

By embedding $[1,H]$ into $\Z/2H\Z$ and extending functions by zero, it suffices to show that
$$
\E_{j,m \in \Z/2H\Z} c_m \prod_{i=1}^k f_i(j + Wmb_i) = o_{w \to \infty}(1)$$
for any functions $f_1,\dots,f_k\colon \Z/2H\Z \to [-1,1]$, where $c_m := \frac{\Lambda_{W,r}(m)-1}{m}$ if $m$ is an integer between $\frac{\eps^2}{2} \frac{H}{W}$ and $\eps^2 \frac{H}{W}$ (identified with an element of $\Z/2H\Z$), and $c_m=0$ otherwise, and we use the averaging notation $\E_{n \in A} f(n) := \frac{1}{|A|} \sum_{n \in A} f(n)$.  
Making the substitution $m = m_1+\dots+m_k$ and $j = n - Wm_1 b_1 - \dots - Wm_k b_k$, we reduce to showing that
$$
\E_{n,m_1,\dots,m_k \in \Z/2H\Z} c_{m_1+\dots+m_k} \prod_{i=1}^k F_i(n,m_1,\dots,m_k) = o_{w \to \infty}(1),$$
where $F_i \colon \Z^{k+1} \to [-1,1]$ is the function
$$ F_i(n,m_1,\dots,m_k) := f_i\left(n + \sum_{j=1}^k W m_j (b_j - b_i)\right).$$
Observe that for each $i=1,\dots,k$, $F_i$ does not depend on the $m_i$ variable.  Applying the triangle inequality in $n$ and the Cauchy-Schwarz inequality $k$ times (as in \cite[(B.7)]{gt-primes}), we see that it suffices to show that
$$
\E_{m^{(0)}_1,\dots,m^{(0)}_k,m^{(1)}_1,\dots,m^{(1)}_k \in \Z/2H\Z} \prod_{\vec \omega \in \{0,1\}^k} c_{\sum_{i=1}^k m^{(\omega_i)}_i} = o_{w \to \infty}(1)$$
where $\vec \omega = (\omega_1,\dots,\omega_k)$.
Writing $h_i := m^{(1)}_i-m^{(0)}_i$ and $x := m^{(0)}_1+\dots+m^{(0)}_k$, we can rewrite the left-hand side as
$$ 
\E_{x,h_1,\dots,h_k \in \Z/2H\Z} \prod_{\vec \omega \in \{0,1\}^k} c_{x + \vec \omega \cdot \vec h},$$
where $\vec \omega \cdot \vec h := \omega_1 h_1 + \dots + \omega_k h_k$,
so by definition of $c_m$, it suffices to show that
$$ \sum_{x,h_1,\dots,h_k \in \Z} \prod_{\vec \omega \in \{0,1\}^k} 1_{\frac{\eps^2}{2} \frac{H}{W} \leq x+\vec \omega \cdot \vec h \leq \eps^2 \frac{H}{W}} \frac{\Lambda_{W,r}(x+\vec \omega \cdot \vec h)-1}{x+\vec \omega \cdot \vec h} = o_{w \to \infty}(H^{k+1}).$$
Using a Riemann sum approximation, it suffices to show that
$$ \sum_{x \in I,h_1 \in J_1,\dots,h_k \in J_k} \prod_{\vec \omega \in \{0,1\}^k} (\Lambda_{W,r}(x+\vec \omega \cdot \vec h)-1) = o_{w \to \infty}((H \log^{-10} H)^{k+1})$$
for all intervals $I,J_1,\dots,J_k \subset [1,H]$ of length $H \log^{-10} H$ (say).  But this follows from the results in \cite{gt-primes}, or more precisely from the localised estimate in \cite[(A.9)]{fgkt}.
\end{proof}

From \eqref{thumb}, the above proposition, and the triangle inequality, we have
$$
\left| \E \sum_{\frac{\eps^2}{2} \frac{H}{W} \leq m \leq \eps^2 \frac{H}{W}} \frac{1}{m} \sum_{j=0\ (a)} \prod_{i=1}^k \lambda(a \n + j + (Wm+r)b_i) 1_{[1,H]}(j+(Wm+r)b_i) \right| \gg \eps H.
$$
Since the expression inside the summation is $O(H)$, we conclude that with probability $\gg_\eps 1$, one has
\begin{equation}\label{will}
\left|\sum_{\frac{\eps^2}{2} \frac{H}{W} \leq m \leq \eps^2 \frac{H}{W}} \frac{1}{m} \sum_{j=0\ (a)} \prod_{i=1}^k \lambda(a \n + j + (Wm+r)b_i) 1_{[1,H]}(j+(Wm+r)b_i) \right| \gg_\eps H.
\end{equation}
Let us condition to this event.  Using our hypothesis that Conjecture \ref{lgi-log} holds (in the form \eqref{loggow-equiv}), together with Markov's inequality, we see that with conditional probability $1 - o_{H \to \infty}(1)$ one also has
\begin{equation}\label{lam}
 \| \lambda \|_{U^{k-1}([a\n,a\n+H] \cap \Z)}= o_{H \to \infty}(1),
\end{equation}
and we condition to this event also.

Replacing $m$ by $Wm+r$, and dropping some negligible boundary terms, we see from \eqref{will} that
$$
\left| \sum_{\frac{\eps^2}{2} H \leq m \leq \eps^2 H} \frac{1_{m = r\ (W)}}{m} \sum_{j = 0\ (a)} \prod_{i=1}^k \lambda(a \n + j + mb_i) 1_{[1,H]}(j+mb_i)\right | \gg_{\eps,W} H.
$$
Since $m=r\ (W)$, and $W$ is a multiple of $a$, we can write $1_{j=0\ (a)}$ as $1_{a\n + j + mb_k = rb_k\ (a)}$.  As $b_1=0$, we may thus write the above estimate in the form
$$
\left| \sum_{\frac{\eps^2}{2} H \leq m \leq \eps^2 H} \frac{1_{m = r\ (W)}}{m} \sum_{j} f_1(j) f_2(j+mb_2) \dots f_k(j+mb_k) \right|
 \gg_{\eps,W} H
$$
for some ($\n$-dependent) functions $f_1, f_2,\dots,f_k\colon \Z \to [-1,1]$ supported on $[1,H]$, with $f_1(j) := \lambda(a\n + j) 1_{[1,H]}(j)$ (the precise values of $f_2,\dots,f_k$ will not be relevant).  Note from \eqref{lam} that
\begin{equation}\label{kak}
 \|f_1\|_{U^{k-1}([1,H] \cap \Z)} = o_{H \to \infty}(1).
\end{equation}
We now dispose of the $m$ weights.  Note that the quantity $ f_1(j) f_2(j+mb_2) \dots f_k(j+mb_k) $ is only non-vanishing when $m=O(H)$, so we may embed the $m$ variable in (say) $\Z/HW\Z$.   We can Fourier expand $m \mapsto 1_{m = r\ (W)}$ into a linear combination of exponential phases $m \mapsto e(sm/W)$ with $s=1,\dots,W$ and coefficients of size $O(1)$.  Similarly, using a standard Fourier expansion (e.g. using\footnote{Alternatively, one can perform a Fourier series expansion of $1_{\frac{\eps^2}{2W} \leq x \leq \frac{\eps^2}{W}} \frac{1}{x}$ on the unit circle.} Fej\'er kernels), one can approximate
$m \mapsto 1_{\frac{\eps^2}{2} H \leq m \leq \eps^2 H} \frac{1}{m}$ on $\Z/HW\Z$ by a linear combination of $O_{\eps,\delta}(1)$ exponential phases $m \mapsto e(sm/HW)$ with $s=1,\dots,H$ and coefficients $O_{\eps,\delta,W}(1/H)$, plus an error whose $\ell^1(\Z/HW\Z)$ norm in $m$ is at most $\delta$, for any given $\delta>0$.  Applying these expansions for $\delta>0$ sufficiently small depending on $\eps,W$, and using the pigeonhole principle, we conclude that
$$
| \sum_m e(sm/HW) \sum_{j} f_1(j) f_2(j+mb_2) \dots f_k(j+mb_k) |
 \gg_{\eps,W} H^2
$$
for some integer $s$, where we now revert to $m$ as taking values in $\Z$ rather than $\Z/HW\Z$.  To deal with the phase $e(sm/HW)$, we write $m$ as a linear combination of $j+mb_{k-1}$ and $j+mb_k$, and conclude (using our assumption $k > 2$) that
$$
| \sum_m \sum_{j} f_1(j) f'_2(j+mb_2) \dots f'_k(j+mb_k) | \gg_{\eps,W} H^2
$$
for some functions $f'_2,\dots,f'_k\colon \Z \to \C$ supported on $[1,H]$ and bounded in magnitude by $1$.  But from the ``generalised von Neumann inequality'' (see e.g. \cite[Lemma 11.4]{tao-vu}, after embedding $[1,H]$ in a cyclic group $\Z/p\Z$ of some prime $p$ between $2H$ and $4H$, say) we have
$$
| \sum_m \sum_{j} f_1(j) f'_2(j+mb_2) \dots f'_k(j+mb_k) | \ll \|f_1\|_{U^{k-1}([1,H] \cap \Z)} 
$$
giving a contradiction to \eqref{kak}.  This concludes the derivation of Conjecture \ref{chow-log} from Conjecture \ref{lgi-log}.

\begin{remark}\label{kch}  An inspection of the above argument shows that if one wishes to establish Conjecture \ref{chow-log} for a specific choice of $k \geq 3$, then it would suffice to establish Conjecture \ref{lgi-log} for $d=k-1$.  In particular, the first open case $k=3$ of Conjecture \ref{chow-log} would follow from a non-trivial bound on the local $U^2$ norms of the Liouville function.
\end{remark}

\begin{remark}  In the spirit of the Elliott conjecture \cite{elliott} (see also \cite{mrt} for a correction to that conjecture), one could more generally consider estimates of the form
$$ \sum_{x/\omega \leq n \leq x} \frac{g_1(n + h_1) \dots g_k(n + h_k)}{n} = o_{\omega \to \infty}(\log \omega) $$
for bounded completely multiplicative functions $g_1,\dots,g_k$.  The weight $\Lambda(m)$ appearing in the above analysis would now be replaced by $\Lambda g_1 \dots g_k(m)$, and so the results on linear equations in primes used in Proposition \ref{cloc} are no longer available.  Nevertheless, one should still be able to deploy a ``transference principle'' to approximate the weight $\Lambda g_1 \dots g_k$ by a small number of ``structured'' functions (such as nilsequences), which should still allow one to derive a suitable generalisation of Conjecture \ref{chow-log} for the $g_1,\dots,g_k$ from Conjecture \ref{lgi-log} (possibly after increasing $d$ to $k$ instead of $k+1$), in the spirit of \cite[Theorem 1.3]{tao-chowla} (which used a ``restriction theorem for the primes'' as a proxy for the transference principle).  We will not pursue this matter here.
\end{remark}

\section{Applying the inverse conjecture for the Gowers norms}\label{gow-sec}

In this section we show how Conjecture \ref{lgi-log} can be deduced from Conjecture \ref{lnc-log}.

Let $d \geq 1$, let $\eps>0$ be sufficiently small depending on $d$, and let $2 \leq H \leq \omega \leq X$ be such that $H$ is sufficiently large depending on $d,\eps$.  We allow implied constants to depend on $d$.
Using the formulation \eqref{loggow-equiv}, our goal is now to show that
$$ \E \| \lambda \|_{U^d([\n,\n+H] \cap \Z)} \ll \eps.$$
Suppose this claim failed, then we must have
\begin{equation}\label{do}
 \| \lambda \|_{U^d([\n,\n+H] \cap \Z)} \gg \eps
\end{equation}
with probability $\gg \eps$.

Suppose that we are in the event that \eqref{do} holds.  Then, by the inverse conjecture for the Gowers norms (\cite[Theorem 1.3]{gtz}), there exists a $d-1$-step (random) nilmanifold ${\mathbf G}/{\mathbf \Gamma}$ from a finite list ${\mathcal M}_{d-1,\eps}$ (each of which is equipped with a smooth Riemannian metric), and a (random) function ${\mathbf F}\colon {\mathbf G}/{\mathbf \Gamma} \to \C$ with Lipschitz constant $O_\eps(1)$ and a random group element $\g \in G$ and random base point $\x_0 \in {\mathbf G}/{\mathbf \Gamma} \to \C$
such that
\begin{equation}\label{lfg}
 |\sum_{h=1}^H \lambda(\n+h) {\mathbf F}( \g^h \x_0 )| \gg_\eps 1.
\end{equation}

By the pigeonhole principle, one can find a \emph{deterministic} $d-1$-step nilmanifold $G/\Gamma$ such that ${\mathbf G}/{\mathbf \Gamma}$ is equal to $G/\Gamma$ with probability $\gg_\eps 1$.  We condition to this event.  Next, we fix a deterministic base point $x_0$ in $G/\Gamma$.  For the random base point $\x_0$, we can write $\x_0 = \g_1 x_0$ for some bounded element $\g_1 \in G$.  We can then write
$$ {\mathbf F}( \g^h \x_0 ) = {\mathbf F}( \g_1 (\g_1^{-1} \g \g_1)^h x_0 ).$$
Replacing $\g$ by $\g_1^{-1} \g \g_1$ and ${\mathbf F}$ by the function $x \mapsto {\mathbf F}(\g_1 x)$, we see that we may assume without loss of generality that $\x_0 = x_0$.  Finally, by the Arzela-Ascoli theorem, the class of Lipschitz functions from $G/\Gamma$ to $\C$ of Lipschitz constant $O_\eps(1)$ is totally bounded in the uniform topology.  Thus, we can restrict the range of possible values of the random function ${\mathbf F}$ to a finite collection of $O_\eps(1)$ deterministic Lipschitz functions without significantly affecting \eqref{lfg}.  By the pigeonhole principle, we can thus find a deterministic Lipschitz function $F\colon G/\Gamma \to \C$ such that
$$
 \left|\sum_{h=1}^H \lambda(\n+h) F( \g^h x_0 )\right| \gg_\eps 1
$$
with probability $\gg_\eps 1$.  In particular,
$$
 \sup_{g \in G} \left|\sum_{h=1}^H \lambda(\n+h) F( g^h x_0 )\right| \gg_\eps 1
$$
with probability $\gg_\eps 1$, which implies that
$$ \E  \sup_{g \in G} \left|\sum_{h=1}^H \lambda(\n+h) F( g^h x_0 )\right| \gg_\eps 1.$$
But this contradicts Conjecture \ref{lnc-log} (in the form \eqref{lognil-equiv}).

\begin{remark}  An inspection of the above argument shows that in order to prove Conjecture \ref{lgi-log} for a specific choice of $d \geq 2$, it suffices to establish Conjecture \ref{lnc-log} with $s=d-1$.  Combining this with Remark \ref{kch}, we see that to establish Conjecture 
\ref{chow-log} for a specific choice of $k \geq 3$, it suffices to establish Conjecture \ref{lnc-log} with $s=d-1$.  In particular, and after performing a Fourier expansion of $1$-step nilsequences $n \mapsto F( g^n x_0 )$, we see that to prove the $k=3$ case of Conjecture \ref{chow-log}, it will suffice to establish the bound
$$ \sum_{X/\omega \leq n \leq X} \frac{\sup_{\alpha \in \R/\Z} |\sum_{h=1}^H \lambda(n+h) e(h\alpha)|}{n}
 = o_{H \to \infty}(H \log \omega).$$
for all $1 \leq H \leq \omega \leq x$.  Bounds of this form are available for very large values of $H$; for instance, the estimates in \cite{zhan} give this bound when $\omega > 1$ is fixed and $H \geq x^{5/8+\eps}$ for any fixed $\eps>0$.  In \cite{mrt} a weaker version of this estimate was established in which $\omega > 1$ is fixed and the supremum in $\alpha$ was outside the summation in $n$.
\end{remark}

\begin{remark}  One can reverse the above arguments, using \cite[Proposition 12.6]{gt-inverseu3} in place of \cite[Theorem 1.3]{gtz}, to show directly that Conjecture \ref{lgi-log} implies Conjecture \ref{lnc-log}; we leave the details of this implication to the interested reader.  This implication of course already follows from the arguments used to prove other components of Theorem \ref{main} in this paper, but this alternate argument is also valid in the absence of logarithmic averaging.
\end{remark}

\section{Constructing a deterministic sequence}\label{detsec}

In this section we show that Conjecture \ref{lnc-log} follows from Conjecture \ref{sar-log}.

Let $s, G/\Gamma$, $x_0$, $F$ be as in Conjecture \ref{lnc-log}; we allow all implied constants to depend on these quantities. By splitting in to real and imaginary parts we may take $F$ to be real-valued.  Let $\eps > 0$.  Our task is to show that
$$
 \sum_{X/\omega \leq n \leq X} \frac{\sup_{g \in G} |\sum_{h=1}^H \lambda(n+h) F( g^{h} x_0 )|}{n} \ll \eps H \log \omega
$$
whenever $1 \leq H \leq \omega \leq X$, and $H$ is sufficiently large depending on $\eps$.  
From\footnote{This result is stated for the M\"obius function in place of the Liouville function, but the arguments extend to the Liouville case; see \cite[\S 6]{gt-mobius}.} \cite[Theorem 1.1]{gt-mobius}, we see that
$$ \sup_{g \in G} \left|\sum_{h=1}^H \lambda(n+h) F( g^{h} x_0 )\right| = o_{H \to \infty}( H )$$
whenever $n \leq H \log H$ (say); in fact the results in \cite{gt-mobius} allow one to improve upon the trivial bound of $O(H)$ by an arbitrary fixed power of $\log H$.  Thus the net contribution of the case $n \leq H \log H$ to \eqref{snam} is negligible, so we may restrict to the case $n > H \log H$.  In this regime, one has $\frac{1}{n+h} = \frac{1}{n} + O( \frac{1}{\log H} \frac{1}{n} )$; the contribution of the error term is negligible (cf. \eqref{ati}), so it suffices to show that
$$
 \sum_{H\log H, X/\omega \leq n \leq X} \sup_{g \in G} \left|\sum_{h=1}^H \frac{\lambda(n+h)}{n+h} F( g^{h} x_0 )\right| \ll \eps H \log \omega
$$

It will suffice to just establish the positive part
\begin{equation}\label{snam}
 \sum_{H\log H, X/\omega \leq n \leq X} \sup_{g \in G} \max\left(\sum_{h=1}^H \frac{\lambda(n+h)}{n+h} F( g^{h} x_0 ),0\right) \ll \eps H \log \omega
\end{equation}
of this estimate, since the full estimate then follows by applying \eqref{snam} for both $F$ and $-F$ and using the triangle inequality.

Suppose for contradiction that the bound \eqref{snam} failed.  Then we can find sequences $H_i, \omega_i, X_i$ with
$$ 1 \leq H_i \leq \omega_i \leq X_i$$
and $H_i \to \infty$ as $i \to \infty$, such that
\begin{equation}\label{nsd}
 \sum_{H_i \log H_i, X_i/\omega_i \leq n \leq X_i} \sup_{g \in G} \max\left(\sum_{h=1}^{H_i} \frac{\lambda(n+h)}{n+h} F( g^{h} x_0 ),0\right) \gg \eps H_i \log \omega_i.
\end{equation}
By sparsifying the sequences $H_i,\omega_i,X_i$ we may assume that
\begin{equation}\label{dead}
 H_{i+1} \geq 100 X_i
\end{equation}
(say) for all $i$.

The quantity $\sup_{g \in G} |\sum_{h=1}^{H_i} \frac{\lambda(n+h)}{n+h} F( g^{h} x_0 )|$ is bounded above by $O(H_i/n)$.  Thus we can find a set $S_i$ of numbers $n$ with $H_i \log H_i, X_i/\omega_i \leq n \leq X_i$ such that
$$ \sum_{n \in S_i} \frac{1}{n} \gg \eps \log \omega_i $$
and such that
\begin{equation}\label{Fail}
 \sup_{g \in G} \sum_{h=1}^{H_i} \frac{\lambda(n+h)}{n+h} F( g^{h} x_0 ) \gg \eps \frac{H_i}{n} 
\end{equation}
for all $n \in S_i$, since the contribution to the left-hand side \eqref{nsd} of those $n$ for which \eqref{Fail} fails can be made to be significantly smaller than the right-hand side of \eqref{nsd} by choosing the implicit constants appropriately.

By a greedy algorithm, we can then find a subset $S'_i$ of $S_i$ that is $H_i$-separated (that is to say, $|n-m| \geq H_i$ for any distinct $n,m \in S'_i$) such that
\begin{equation}\label{song}
 \sum_{n \in S'_i} \frac{1}{n} \gg \frac{\eps}{H_i} \log \omega_i.
\end{equation}
For each $n \in S'_i$, we can find a group element $g_n \in G$ such that
\begin{equation}\label{sing}
 \sum_{h=1}^{H_i} \frac{\lambda(n+h)}{n+h} F( g_n^{h} x_0 ) \gg \eps \frac{H_i}{n}.
\end{equation}
If we now set $f\colon \Z \to \R$ to be the function defined by setting
$$ f(n+h) := F( g_n^{h} x_0 ) $$
whenever $n \in S'_i$ and $1 \leq h \leq H_i$ for some $i$, and $f(m)=0$ for all other $m$, we see that $f$ is well-defined because all the intervals $\{ n+1,\dots,n+H_i\}$ with $n \in S'_i$ and $i \geq 1$ are disjoint, thanks to \eqref{dead} and the $H_i$-separation of the $S'_i$.  

Summing \eqref{sing} over all $n \in S'_i$ and using \eqref{song}, we conclude that
$$ \sum_{H_i \log H_i, X_i/\omega_i \leq n \leq 2 X_i} \frac{\lambda(n)}{n} f(n) \gg \eps^2 \log \omega_i.$$
On the other hand, if $f$ is deterministic, then Conjecture \ref{sar-log} gives
$$ \sum_{H_i \log H_i, X_i/\omega_i \leq n \leq 2 X_i} \frac{\lambda(n)}{n} f(n) = o_{\omega_i \to \infty}(\log \omega_i)$$
(one can divide here into two cases, depending on whether $\log \frac{2 X_i}{H_i \log H_i}$ is smaller than (say) $\sqrt{\log \omega_i}$ or not).  Thus it will suffice to show that the sequence $f$ is deterministic.

Since $F$ is bounded, $f$ takes values in a compact interval $[-C,C]$.  Consider the compact space
$$ [-C,C]^\Z = \{ (y_n)_{n \in \Z}: y_n \in [-C,C] \forall n \in \Z \}$$
which we endow with the shift
$$ T( y_n)_{n \in \Z} := (y_{n+1})_{n \in \Z}$$
and metric
$$ d( (x_n)_{n \in \Z}, (y_n)_{n \in \Z} ) := \sup_{n \in \Z} 2^{-|n|} |x_n-y_n|.$$
We can identify $f$ with a point $y_0 := (f(n))_{n \in \Z}$ in $[-C,C]^\Z$.  We let $Y = \overline{\{ T^n y_0: n \in \Z\}}$ be the orbit closure of $y_0$ in $[-C,C]^\Z$, then $(Y,T)$ is a topological dynamical system.  If we let $F_0\colon Y \to \R$ be the function
$$ F_0((y_n)_{n \in \Z}) := y_0$$
then $F_0$ is continuous and 
$$ f(n) = F_0(T^n y_0)$$
for all $n \in \Z$.  Thus, to show that $f$ is deterministic, it suffices to show that $(Y,T)$ has zero topological entropy.  That is to say, for any fixed $\eps>0$ and any sufficiently large $N$, we should be able to cover $Y$ by at most $\exp( O(\eps N) )$ balls of radius $O(\eps)$ in the metric
$$ d_N( x, y ) := \max_{0 \leq i \leq N} d( T^i x, T^i y)$$
or equivalently
$$ d_N( (x_n)_{n \in \Z}, (y_n)_{n \in \Z} ) = \sup_{n \in \Z} 2^{- \max( -n, 0, n-N )} |x_n-y_n|.$$

Observe that if two sequences $(x_n)_{n \in \Z}, (y_n)_{n \in \Z}$ are such that $x_n=y_n+O(\eps)$ for all $-N \leq n \leq 2N$, then (for $N$ sufficiently large depending on $\eps$) we have $d_N( (x_n)_{n \in \Z}, (y_n)_{n \in \Z} )$.  Thus it suffices to find a collection ${\mathcal S}_{\eps,N}$ of finite sequences $(x_h)_{-N \leq h \leq 2N}$ of cardinality $\exp( O(\eps N))$ with the property that for every $n \in \Z$, there exists a sequence $(x_h)_{-N \leq h \leq 2N}$ in ${\mathcal S}_{\eps,N}$ such that
$$ f(n+h) = x_h + O(\eps)$$
for all $-N \leq h \leq 2N$.

Observe that if we can prove this claim for a given value of $N$, then we automatically obtain the claim for any larger $N' \geq N$ (with a slightly worse implicit constant), by covering the interval $[-N',2N']$ by $O(N'/N)$ translates of $[-N,2N]$.  In particular, it will suffice to verify the claim with $N = \lfloor H_{i_0} / 10 \rfloor$ for $i_0$ sufficiently large depending on $\eps$.

We may remove from consideration those $n$ for which $|n| \leq 2N$, since these cases can be accommodated simply by adding the sequences $(f(n+h))_{-N \leq h \leq 2N}$ for $|n| \leq 2N$ to ${\mathcal S}_{\eps,N}$, which only increases the cardinality of that family by a negligible amount.  If $n < -2N$ then one has $f(n+h)=0$ for all $-N \leq h \leq 2N$, and this case can be accommodated by adding the zero sequence $(0)_{-N \leq h \leq 2N}$ to ${\mathcal S}_{\eps,N}$.  Thus we may assume that $n > 2N$.

Recall that the function $f$ is only supported on the union of the intervals $\{m+1,\dots,m+H_i\}$ with $i \geq 1$ and $m \in S'_i$, so in particular $H_i \log H_i \leq m \leq X_i$.  Since $n > 2N$, such an interval can only intersect the interval $\{n-N,\dots,n+2N\}$ if one has
$$H_i \log H_i \ll n \ll X_i;$$
in particular there is at most one choice of $i$ in which this can occur.  Since $n \geq 2N = 2 \lfloor H_{i_0}/10\rfloor$, we conclude from \eqref{dead} that $i \geq i_0$, so in particular $H_i \geq 10 N$.  In particular, each interval $\{n-N, \dots, n+2N \}$ meets at most two of the intervals $\{m+1,\dots,m+H_i\}$.  It will now suffice to exhibit a set ${\mathcal S}'_{\eps,N}$ of finite sequences $(x_n)_{-N \leq h \leq 2N}$ of cardinality $O(\exp(O(\eps N)))$ with the property that for any $i \geq i_0$, any $m \in S'_i$, and sub-interval $\{n-N,\dots,n+2N\}$ of $\{m+1,\dots,m+H_i\}$, there exists a sequence $(x_n)_{-N \leq h \leq 2N}$ in ${\mathcal S}'_{\eps,N}$ for which
$$ f(n+h) = x_h + O(\eps)$$
for all $-N \leq h \leq 2N$. Indeed, one can now set ${\mathcal S}_{\eps,N}$ to be the collection of all sequences $(x_n)_{-N \leq h \leq 2N}$ formed by concatenating at most two subsequences of sequences in ${\mathcal S}'_{\eps,N}$, together with some blocks of zeroes; the cardinality of ${\mathcal S}_{\eps,N}$ is $O( N^{O(1)} | {\mathcal S}'_{\eps,N} |^2 )$, which will be at most $\exp(O(\eps N))$ if $N$ is large enough.

It remains to exhibit ${\mathcal S}'_{\eps,N}$.  If $n, m$ are as above, then
$$ f(n+h) = F( g_m^{n+h-m} x_0 ) $$
for $-N \leq h \leq 2N$.  In particular, there exists a polynomial sequence $g_n\colon \Z \to G$, that is to say a sequence of the form
$$ g_n(h) = g_{n,0} g_{n,1}^h g_{n,2}^{\binom{h}{2}} \dots g_{n,s}^{\binom{h}{s}}$$
where $g_{n,i} \in G_i$ for $i=0,\dots,s$, and $G = G_0 = G_1 \geq G_2 \geq \dots \geq G_s$ is the lower central series of $G$, such that 
$$f(n+h) = F( g_n(h) \Gamma )$$
for $-N \leq h \leq 2N$.  Currently we have $g_{n,0} = g_m^{n-m}$, $g_{n,1} = g_m$, and all other coefficients trivial; however we shall shortly consider more general polynomial sequences in which the higher coefficients $g_{n,2},\dots,g_{n,s}$ are allowed to be non-trivial.

The coefficients $g_{n,i}$ of an arbitrary polynomial sequence $g_n$ can be unbounded.  However, any such sequence $g_n$ may be factorised as
$$ g_n = \tilde g_n \gamma_n$$
where $\tilde g_n$ is a polynomial sequence with coefficients taking values in a compact set (depending only on $G,\Gamma$) and $\gamma_n$ is a polynomial sequence with coefficients in $\Gamma$; see \cite[Lemma C.1]{gtz} for a proof.  In particular, we have $g_n(h) \Gamma = \tilde g_n(h) \Gamma$ for any $h$, and hence
$$f(n+h) = F( \tilde g_{n,0} \tilde g_{n,1}^h \tilde g_{n,2}^{\binom{h}{2}} \dots \tilde g_{n,s}^{\binom{h}{s}} \Gamma )$$
for all $-N \leq h \leq 2N$ and some coefficients $\tilde g_{n,0},\dots,\tilde g_{n,s}$ in some fixed compact subset $K$ of $G$.

Let $A$ be a large constant depending on $G,\Gamma$ to be chosen later.  From many applications of the Baker-Campbell-Hausdorff formula (which is a polynomial formula in a connected, simply connected nilpotent Lie group such as $G$), we see that if we modify each of the coefficients $\tilde g_{n,0},\dots,\tilde g_{n,s}$ by at most $O(N^{-A})$ (after endowing $G$ with some smooth left-invariant Riemannian metric), then the quantities $\tilde g_{n,0} \tilde g_{n,1}^h \tilde g_{n,2}^{\binom{h}{2}} \dots \tilde g_{n,s}^{\binom{h}{s}}$ for $-N \leq h \leq 2N$ only change by $O( N^{-A+O(1)})$ in the $G$ metric.  In particular, if we select a maximal $N^{-A}$-separated net $\Sigma$ of $K$, and let $g'_{n,i}$ be $\tilde g_{n,i}$ rounded to the nearest element of $\Sigma$ (breaking ties arbitrarily), then from the Lipschitz nature of $F$ we have
$$f(n+h) = F( g'_{n,0} (g'_{n,1})^h (g'_{n,2})^{\binom{h}{2}} \dots (g'_{n,s})^{\binom{h}{s}} \Gamma ) + O(N^{-A+O(1)}).$$
If we choose $A$ large enough, then the error term $O(N^{-A+O(1)})$ is $O(\eps)$.  If we now set ${\mathcal S}'_{\eps,N}$ to be the collection of all sequences of the form
$$ (F( g_0 g_1^h g_2^{\binom{h}{2}} \dots g_s^{\binom{h}{s}} \Gamma ))_{-N \leq h \leq 2N}$$
with $g_0,\dots,g_s \in \Gamma$, then ${\mathcal S}'_{\eps,N}$ has cardinality $O( N^{O(A)}) = O( \exp( O( \eps N ) ) )$ for $N$ large enough, and the claim follows.

\begin{remark}\label{Entr}  The main fact that was used in the above argument is that the collection of nilsequences $n \mapsto F( g^n x_0 )$, where $F$ is a Lipschitz function on a nilmanifold $G/\Gamma$ of ``bounded complexity'', $g \in G$, and $x_0 \in G/\Gamma$, has ``uniform zero entropy'' in the sense that for any $\eps>0$ and any $N$ sufficiently large depending on $\eps$, the set of sequences formed from evaluating an arbitrary nilsequence in this collection at $N$ consecutive values has a metric entropy of $O( \exp(O(\eps N)))$ at scale $\eps>0$.  This is stronger than asserting that each individual nilsystem $(G/\Gamma, x \mapsto gx)$, $g \in G$ has zero entropy, as one needs to control the metric entropy of the set of sequences arising from \emph{arbitrary} shifts $g$, rather than just one shift at a time.  On the other hand, if all one is interested in is deducing Conjecture \ref{lgi-log} from Conjecture \ref{sar-log}, it is likely that one does not need the full strength of the inverse conjecture in \cite{gtz}, and in particular one does not need to introduce the notion of a nilmanifold or nilsequence at all.  Instead, one can rely on ``soft'' inverse theorems in which the role of nilsequences are replaced by those of dual functions (see e.g. \cite[\S 11.4]{tao-vu}), in which case the task is basically reduced to establishing that the collection of dual functions also has ``uniform zero entropy'' in a certain sense.  This in turn should be provable using some sort of random sampling argument to show that the dual function of a given function $f$ is almost completely controlled by the values of $f$ at some sparse random subset of the domain.  We will however not attempt to formalise these arguments here.
\end{remark}

\end{document}